\documentclass[12pt, a4paper]{amsart}
\usepackage{amsmath}
\usepackage{latexsym,amsfonts,amssymb,amsmath,amsxtra,bbm,color,mathrsfs}
\usepackage{amscd}
\usepackage{mathdots}
\allowdisplaybreaks[4]

\usepackage{braket}

\usepackage{hyperref}
\usepackage{pdfsync}
\usepackage[all]{xypic}

\usepackage{verbatim}

\newcommand{\CE}{{\mathcal {E}}}

\newcommand{\CJ}{{\mathcal {J}}}

\newcommand{\RB}{{\mathrm {B}}}
\newcommand{\RC}{{\mathrm {C}}}

\newcommand{\RG}{{\mathrm {G}}}
\newcommand{\RH}{{\mathrm {H}}}
\newcommand{\RI}{{\mathrm {I}}}
\newcommand{\RJ}{{\mathrm {J}}}

\newcommand{\RL}{{\mathrm {L}}}

\newcommand{\RN}{{\mathrm {N}}}

\newcommand{\RP}{{\mathrm {P}}}

\newcommand{\RT}{{\mathrm {T}}}
\newcommand{\RU}{{\mathrm {U}}}

\newcommand{\RZ}{{\mathrm {Z}}}

\newcommand{\GL}{{\mathrm{GL}}}

\renewcommand{\Re}{{\mathrm{Re}}}

\newcommand{\Z}{\mathbb{Z}}
\newcommand{\C}{\mathbb{C}}
\newcommand{\R}{\mathbb R}

\newcommand{\K}{\mathbb{K}}
\newcommand{\M}{\mathbf{M}}

\newcommand{\be}{\begin {equation}}
\newcommand{\ee}{\end {equation}}
\newcommand{\bee}{\begin {equation*}}
\newcommand{\eee}{\end {equation*}}

\theoremstyle{Theorem}

\theoremstyle{Theorem}

\theoremstyle{Theorem}

\theoremstyle{Theorem}
\newtheorem{prp}{Proposition}[section]

\newtheorem{lemp}[prp]{Lemma}
\newtheorem{thmp}[prp]{Theorem}

\theoremstyle{Plain}

\theoremstyle{Definition}

\newtheorem{dfnp}[prp]{Definition}

\begin{document}

	\title[Archimedean Modular Symbols]{Archimedean modular symbols for Rankin-Selberg convolutions of $\GL_n(\C)\times \GL_n(\C)$ and $\GL_n(\C)\times \GL_{n+1}(\C)$}
	
	\author[Y. Jin]{Yubo Jin}
	\address{Institute for Advanced Study in Mathematics, Zhejiang University\\
		Hangzhou, 310058, China}\email{yubo.jin@zju.edu.cn}


	
	\author[Z. Tian]{Zeyu Tian}
	\address{School of Mathematical Sciences,  Zhejiang University\\
		Hangzhou, 310058, China}\email{tianzeyu@zju.edu.cn}
	
	\date{\today}

	\maketitle
	
	\begin{abstract}
		We explicitly calculate the archimedean modular symbol for Rankin-Selberg convolutions of $\GL_n(\C)\times\GL_n(\C)$ and $\GL_n(\C)\times\GL_{n+1}(\C)$, for all generic cohomological representations. 
	\end{abstract}
	
	\section{Introduction and main results}
	
	This paper continues the study of archimedean modular symbols for Rankin-Selberg convolutions of $\GL_n\times\GL_n$ and $\GL_n\times\GL_{n+1}$ defined in \cite{JLSa, LLS24} over the complex archimedean local field. The non-vanishing of archimedean modular symbols, also known as the non-vanishing hypothesis, was first proved by \cite{Sun} for essentially tempered representations in the $\GL_n\times\GL_{n+1}$ case and partially in \cite{DX} for the $\GL_n\times\GL_n$ case. In \cite{JLSa, LLS24}, the archimedean period relations for both cases are established, which imply the non-vanishing hypothesis in full generality. The archimedean period relation is also a pivotal ingredient in establishing period relations for special values of L-functions in \cite{JLLS26, LLS24}. However, the archimedean modular symbols are not explicitly calculated in the above-mentioned papers. It is the aim of the present paper to provide an explicit calculation as a complement to \cite{JLSa, LLS24}. The result of the present paper will also be used in \cite{JLLS26} to study rationality of special values of Rankin-Selberg L-functions for $\GL_n\times\GL_n$.

	The explicit calculation for the $\GL_n\times\GL_{n+1}$ case is already carried out in \cite{HMN}, for essentially tempered representations. The restriction there to essentially tempered representations is imposed for the global application to cuspidal regular algebraic representations and can be removed without difficulty. Therefore the main focus of this paper is the $\GL_n\times\GL_n$ case. We acknowledge that this work is substantially based on the calculation of archimedean Rankin-Selberg integrals by Ishii-Miyazaki \cite{IM} and the strategy of calculating archimedean modular symbols in \cite{HMN}. 
	
	We introduce various notions and state our main results in the subsequent subsections. For any positive integer $n$, $\GL_n$ is the general linear group over $\Z$. Let $\RB_n=\RT_n\RN_n$ be the upper-triangular Borel subgroup of $\GL_n$ with $\RT_n$ the diagonal torus and $\RN_n$ the unipotent radical. Likewise, let $\overline{\RB}_n=\RT_n\overline{\RN}_n$ be the lower-triangular Borel subgroup with unipotent radical $\overline{\RN}_n$. The centre of $\GL_n$ is denoted by $\RZ_n$. Denote by $1_n$ the identity matrix of size $n$.
	
	Let $\K$ be an archimedean local field which is topologically isomorphic to $\C$. Denote by $\CE_\K$ the set of continuous field isomorphisms from $\K$ to $
    \C$. We fix a continuous field isomorphism $\iota:\K\xrightarrow{\sim}\C \in \CE_\K$ and denote by $\overline{\iota}$ the composition of $\iota$ with the complex conjugation. Let $K_n:=\RU(n)$, to be viewed as a maximal compact subgroup of $\GL_n(\K)$ via
	\[
	\RU(n)\subset\GL_n(\C)\xrightarrow{\iota^{-1}}\GL_n(\K).
	\] 
	Put $\widetilde{K}_n:=K_n\cdot\RZ_n(\K)$. Let $\mathfrak{g}_n$, $\mathfrak{k}_n$, $\widetilde{\mathfrak{k}}_n$ denote the complexified Lie algebras of $\GL_n(\K)$, $K_n$, $\widetilde{K}_n$ respectively.

	\subsection{Generic cohomological representations}
	
	Fix a positive integer $n$. As usual, we say that a weight
	\[
	\mu:=(\mu^{\iota};\mu^{\overline{\iota}}):=(\mu_1^{\iota},\mu_2^{\iota},\dots,\mu_n^{\iota};\mu_1^{\overline{\iota}},\mu_2^{\overline{\iota}},\dots,\mu_n^{\overline{\iota}})\in\Z^n\times\Z^n
	\]
	is dominant if
	\[
	\mu_1^{\iota}\geq\mu_2^{\iota}\geq\cdots\geq\mu_n^{\iota},\quad\mu_1^{\overline{\iota}}\geq\mu_2^{\overline{\iota}}\geq\cdots\geq\mu_n^{\overline{\iota}}.
	\]
	
	Let $\mu$ be such a dominant weight and denote by $F_{\mu}$ the unique (up to isomorphism) irreducible holomorphic finite-dimensional representation of $\GL_n(\K\otimes_{\R}\C)$ of highest weight $\mu$. As a consequence of Delorme's lemma (\cite[Theorem III.3.3]{BW}), there is a unique (up to isomorphism) generic irreducible cohomological Casselman-Wallach representation $\pi_{\mu}$ of $\GL_n(\K)$ such that the total relative Lie algebra cohomology
	\[
	\RH^{\ast}\left(\mathfrak{g}_n,\widetilde{K}_n;F_{\mu}^{\vee}\otimes\pi_{\mu}\right)\neq\{0\}.
	\]
	Here and henceforth, a superscript `$^{\vee}$' over a finite-dimensional representation or Casselman-Wallach representation indicates its contragredient representation. 
	
	More explicitly, the representation $\pi_{\mu}$ can be realized as the principal series representation
	\[
	\pi_{\mu}=\mathrm{Ind}^{\GL_n(\K)}_{\overline{\RB}_n(\K)}\left(\iota^{\widetilde{\mu}_1^{\iota}}\overline{\iota}^{\widetilde{\mu}_n^{\overline{\iota}}}\otimes\cdots\otimes\iota^{\widetilde{\mu}_n^{\iota}}\overline{\iota}^{\widetilde{\mu}_1^{\overline{\iota}}}\right),
	\]
	where
	\[
	\widetilde{\mu}_i^{\iota}:=\mu_i^{\iota}+\frac{n+1-2i}{2},\qquad\widetilde{\mu}_i^{\overline{\iota}}:=\mu_i^{\overline{\iota}}+\frac{n+1-2i}{2},\qquad 1\leq i\leq n,
	\]
	and $\iota^a\overline{\iota}^b$ indicates the character $z\mapsto\iota(z)^{a-b}(\iota(z)\overline{\iota}(z))^b$ of $\K^{\times}$ for $a,b\in\C$ with $a-b\in\Z$. Here and henceforth, `$\mathrm{Ind}$' indicates the normalized smooth induction. Delorme's lemma furthermore implies that
	\[
	\RH^{i}\left(\mathfrak{g}_n,\widetilde{K}_n;F_{\mu}^{\vee}\otimes\pi_{\mu}\right)=\{0\},\qquad\text{if }i<b_n:=\frac{n(n-1)}{2}
	\]
	and
	\[
	\RH^{b_n}\left(\mathfrak{g}_n,\widetilde{K}_n;F_{\mu}^{\vee}\otimes\pi_{\mu}\right)\cong\C.
	\]
	
	Fix a unitary character 
	\[
	\psi_{\R}:\R\to\C^{\times},\qquad x\mapsto e^{-2\pi\mathrm{i}x}
	\]
	with $\mathrm{i}:=\sqrt{-1}$. Define a character
	\[
	\psi_{\K}:\K\to\C^{\times},\qquad x\mapsto\psi_{\R}\left(\sum_{\iota\in\CE_{\K}}\iota(x)\right),
	\]
	which induces a unitary character
	\[
	\psi_{n,\K}:\RN_n(\K)\to\C^{\times},\qquad [x_{i,j}]_{1\leq i,j\leq n}\mapsto\psi_{\K}\left(\sum_{i=1}^{n-1}x_{i,i+1}\right).
	\]
	
	In the sequel, we let $n'=n$ or $n'=n+1$ and consider dominant weights $\mu\in\Z^{n'}\times\Z^{n'}$, $\nu\in\Z^{n}\times\Z^{n}$. Let $\pi_{\mu},\pi_{\nu}$ be associated generic cohomological representations and  $F_{\mu},F_{\nu}$ be associated irreducible holomorphic finite-dimensional representations. We fix generators
	\[
	\lambda_{\mu}\in\mathrm{Hom}_{\RN_{n'}(\K)}(\pi_{\mu},\psi_{n',\K}),\qquad\lambda_{\nu}\in\mathrm{Hom}_{\RN_{n}(\K)}(\pi_{\nu},\overline{\psi_{n,\K}})
	\]
	for the one-dimensional spaces of Whittaker functionals in Section \ref{whittaker}.

	\subsection{Degenerate principal series representations}
	
	Let $\RP_n(\K)$ be the parabolic subgroup of $\GL_n(\K)$ consisting of all the matrices whose last row equals $\begin{bmatrix}
		0 & \cdots & 0 & a
	\end{bmatrix}$ for some $a\in\K^{\times}$. Denote by $\eta$ the product of central characters of $\pi_{\mu}$ and $\pi_{\nu}$. For a character $\chi:\K^{\times}\to\C^{\times}$, we consider the degenerate principal series representations (smooth induction)
	\begin{equation}
		\begin{aligned}
			I_{\eta,\chi}:=&\left(\mathrm{Ind}^{\GL_n(\K)}_{\RP_n(\K)}\left(\mathbf{1}\otimes|\cdot|_{\K}^{\frac{n}{2}}\eta^{-1}\chi^{-n}\right)\right)\otimes\left((|\cdot|_{\K}^{-\frac{1}{2}}\cdot\chi)\circ\det\right)
		\end{aligned}
	\end{equation}
	where $\mathbf{1}$ is the trivial representation of $\GL_{n-1}(\K)$ and $\det$ indicates the determinant character. When $n=1$, we understand $I_{\eta,\chi}=\eta^{-1}$ as a character. 
	As usual, $|\cdot|_{\K}$ is the normalized absolute value on $\K$. Let
	\begin{equation}
		\mathrm{d}\eta:=\{\eta_{\iota};\eta_{\overline{\iota}}\}\in\C\times\C\qquad\mathrm{d}\chi:=\{\chi_{\iota};\chi_{\overline{\iota}}\}\in\C\times\C
	\end{equation}
	denote the complexified differentiations of $\eta$ and $\chi$ respectively. When $n'=n$, we impose the condition that $\mathrm{d}\eta\in\Z\times\Z$, $\mathrm{d}\chi\in\Z\times\Z$ and
	\begin{equation}
		\label{CM}
		\eta_{\iota}+n\chi_{\iota}\leq 0\qquad \eta_{\overline{\iota}}+n\chi_{\overline{\iota}}\geq n.
	\end{equation}
	This is labeled as Case ($\pm$) in \cite{JLSa}. We remark that our discussion also applies to the case that
	\begin{equation}\label{CM'}
		\eta_{\iota}+n\chi_{\iota}\geq n\qquad \eta_{\overline{\iota}}+n\chi_{\overline{\iota}}\leq 0.
	\end{equation}
	See Section \ref{sec:complexconjugation}.
	
	Denote by $F_{\eta,\chi}$ the irreducible holomorphic finite-dimensional representation of $\GL_n(\K\otimes_{\R}\C)$ whose infinitesimal character equals that of $I_{\eta,\chi}$.  As a consequence of Delorme's lemma \cite[Theorem III.3.3]{BW}, we have that
	\[
	\RH^i\left(\mathfrak{g}_n,\widetilde{K}_n;F_{\eta,\chi}^{\vee}\otimes I_{\eta,\chi}\right)=\{0\},\qquad\text{if }i<c_n:=n-1,
	\]
	and
	\[
	\RH^{c_n}\left(\mathfrak{g}_n,\widetilde{K}_n;F_{\eta,\chi}^{\vee}\otimes I_{\eta,\chi}\right)\cong\C.
	\]
	
	For an algebraic character $\chi$, denote by $F_{\chi}:=\C$ the holomorphic character of $\GL_n(\K\otimes_{\R}\C)$ extending the character $\chi\circ\det:\GL_{n}(\K)\to\C^{\times}$.

	\begin{dfnp}
		An algebraic character $\chi$ of $\K^{\times}$ is said to be $(\mu,\nu)$-balanced if
		\[
		\begin{cases}
			\mathrm{Hom}_{\GL_{n}(\K\otimes_{\R}\C)}\left(F_{\mu}^{\vee}\otimes F_{\nu}^{\vee}\otimes F_{\chi}^{\vee},\C\right)\neq \{0\},&\text{when }n'=n+1;\\
			\mathrm{Hom}_{\GL_{n}(\K\otimes_{\R}\C)}\left(F_{\mu}^{\vee}\otimes F_{\nu}^{\vee}\otimes F_{\eta,\chi}^{\vee},\C\right)\neq \{0\},&\text{when }n'=n.
		\end{cases}
		\]
		Denote the set of all $(\mu,\nu)$-balanced characters by $\RB(\mu,\nu)$.
	\end{dfnp}
	
	From now on, we assume that $\chi\in\RB(\mu,\nu)$. The condition on $\chi$ being $(\mu,\nu)$-balanced is determined in \cite[Lemma 2.6]{JLSa}. Note that the above $\mathrm{Hom}$ spaces have dimension one and a generator
	\[
	\phi_{\mu,\nu,\chi}\in	\begin{cases}
		\mathrm{Hom}_{\GL_{n}(\K\otimes_{\R}\C)}\left(F_{\mu}^{\vee}\otimes F_{\nu}^{\vee}\otimes F_{\chi}^{\vee},\C\right),&\text{when }n'=n+1;\\
		\mathrm{Hom}_{\GL_{n}(\K\otimes_{\R}\C)}\left(F_{\mu}^{\vee}\otimes F_{\nu}^{\vee}\otimes F_{\eta,\chi}^{\vee},\C\right),&\text{when }n'=n.
	\end{cases}
	\]
	is normalized as in \cite[Section 1.3]{JLSa}, with $\iota$ and $\overline{\iota}$ interchanged when \eqref{CM'} holds.

	\subsection{Archimedean modular symbols}
	
	Fix the Haar measure $\mathrm{d}x$ on $\K$ to be the self-dual measure with respect to $\psi_{\K}$. The Haar measure on $\RN_{n}(\K)$ is fixed to be the product of the Haar measures on $\K$, and the measure on $\RZ_n(\K)=\K^{\times}$ is fixed to be $\frac{1}{\pi}\frac{\mathrm{d}x}{|x|_{\K}}$. Denote by $\mathfrak{M}_{n}$ the one-dimensional complex vector space of left (or right) invariant measures on $\GL_{n}(\K)$. Any invariant measure $\mathrm{d}g$ on $\GL_{n}(\K)$ induces quotient measures on $\RN_{n}(\K)\backslash\GL_{n}(\K)$ and $\RZ_{n}(\K)\RN_{n}(\K)\backslash\GL_{n}(\K)$. We will denote both these quotient measures by $\overline{\mathrm{d}}g$. 
	
	We have identifications
	\begin{equation}
		\label{measure}
		\begin{aligned}
			\mathfrak{M}_{n}=\begin{cases}
				\{\text{invariant measures on }\GL_{n}(\K)/K_{n}\}, & n'=n+1,\\
				\{\text{invariant measures on }\GL_{n}(\K)/\widetilde{K}_{n}\}, & n'=n,
			\end{cases}
		\end{aligned}
	\end{equation}
	by fixing the Haar measure on $K_{n}$ with total volume $1$. Denote by $\mathfrak{O}_{n+1,n}$ (resp. $\mathfrak{O}_{n,n}$) the one-dimensional space of $\GL_{n}(\K)$-invariant sections of the orientation line bundle of $\GL_{n}(\K)/K_{n}$ (resp. $\GL_{n}(\K)/\widetilde{K}_{n}$) with complex coefficients, respectively. Recall that there are canonical identifications
	\begin{equation}
		\label{orientationid}
		\begin{aligned}
			\mathfrak{M}_n&=\wedge^{n^2}(\mathfrak{g}_n/\mathfrak{k}_n)^{\vee}\otimes\mathfrak{O}_{n+1,n}, &\qquad &n'=n+1,\\
			\mathfrak{M}_n&=\wedge^{n^2-1}(\mathfrak{g}_n/\widetilde{\mathfrak{k}}_n)^{\vee}\otimes\mathfrak{O}_{n,n}, &\qquad & n'=n.
		\end{aligned}
	\end{equation}
	See \cite[Section 3.1]{LLS24} for more details. Here and henceforth, a superscript `$^\vee$' over a vector space indicates its dual space.
	
	\subsubsection{The case $n'=n+1$}

	We view $\GL_{n}$ as an algebraic subgroup of $\GL_{n+1}$ via the embedding
	\[
	g\mapsto\begin{bmatrix}
		g & 0\\
		0 & 1
	\end{bmatrix}.
	\]
	For every $s\in\C$, denote by $\chi_s:=\chi\cdot|\cdot|^s_{\K}$. We have the Rankin-Selberg integral
	\[
	\RZ(f,f';\chi_s,\mathrm{d}g):=\int_{\RN_{n}(\K)\backslash\GL_{n}(\K)}\langle\lambda_{\mu},g.f\rangle\langle\lambda_{\nu},g.f'\rangle\chi_s(\det g)\overline{\mathrm{d}}g,
	\]
	where $f\in\pi_{\mu}$, $f'\in\pi_{\nu}$ and $\mathrm{d}g\in\mathfrak{M}_{n}$. It converges absolutely when $\mathrm{Re}(s)$ is sufficiently large and has a meromorphic continuation to the whole complex plane in $s$. By \cite{J09}, the normalized Rankin-Selberg integral
	\[
	\RZ^{\circ}(f,f';\chi_s,\mathrm{d}g):=\frac{1}{\RL(s+\frac{1}{2},\pi_{\mu}\times\pi_{\nu}\times\chi)}\RZ(f,f';\chi_s,\mathrm{d}g)
	\]
	is holomorphic at $s=0$ and specializing to $s=0$ yields a non-zero linear functional
	\[
	\RZ_{\mu,\nu,\chi}^{\circ}\in\mathrm{Hom}_{\GL_{n}(\K)}\left(\pi_{\mu}\widehat{\otimes}\pi_{\nu}\otimes(\chi\circ\det),\mathfrak{M}_{n}^{\vee}\right).
	\]

	Define a vector space
	\[
	\begin{aligned}
		\RH_{\mu,\nu,\chi}:=\RH^{b_{n+1}}\left(\mathfrak{g}_{n+1},\widetilde{K}_{n+1};F_{\mu}^{\vee}\otimes\pi_{\mu}\right)&\otimes\RH^{b_{n}}\left(\mathfrak{g}_{n},\widetilde{K}_{n};F_{\nu}^{\vee}\otimes\pi_{\nu}\right)\\
		&\otimes\RH^0\left(\mathfrak{g}_{n},\widetilde{K}_{n};F_{\chi}^{\vee}\otimes(\chi\circ\det)\right).
	\end{aligned}
	\]
	We define the archimedean modular symbol, which is a linear functional
	\[
	\wp_{\mu,\nu,\chi}:\RH_{\mu,\nu,\chi}\otimes\mathfrak{O}_{n+1,n}\to\C
	\]
	as the composition
	\[
	\begin{aligned}
		&\RH_{\mu,\nu,\chi}\otimes\mathfrak{O}_{n+1,n}\\
		\xrightarrow{\mathrm{res}}	& \RH^{n^2}\left(\mathfrak{g}_{n},K_{n};F_{\mu}^{\vee}\otimes F_{\nu}^{\vee}\otimes F_{\chi}^{\vee}\otimes\pi_{\mu}\widehat{\otimes}\pi_{\nu}\otimes(\chi\circ\det)\right)\otimes\mathfrak{O}_{n+1,n}\\
		\xrightarrow{\phi_{\mu,\nu,\chi}\otimes\RZ_{\mu,\nu,\chi}^{\circ}}&\RH^{n^2}\left(\mathfrak{g}_{n},K_{n};\mathfrak{M}_{n}^{\vee}\right)\otimes\mathfrak{O}_{n+1,n}=\C.
	\end{aligned}
	\]
	Here the first map is the restriction in cohomology; the second map is induced by the balanced map $\phi_{\mu,\nu,\chi}$ and the normalized Rankin-Selberg integral $\RZ_{\mu,\nu,\chi}^{\circ}$; the final identification is due to \eqref{orientationid}.
	
	\subsubsection{The case $n'=n$}
	
	For every $s\in\C$, we have the Rankin-Selberg integral
	\[
	\RZ(f,f',\varphi_s;\chi_s,\mathrm{d}g):=\int_{\RZ_n(\K)\RN_{n}(\K)\backslash\GL_{n}(\K)}\langle\lambda_{\mu},g.f\rangle\langle\lambda_{\nu},g.f'\rangle\varphi_s(g)\overline{\mathrm{d}}g,
	\]
	where $f\in\pi_{\mu}$, $f'\in\pi_{\nu}$, $\varphi_s\in I_{\eta,\chi_s}$ and $\mathrm{d}g\in\mathfrak{M}_{n}$. It converges absolutely when $\mathrm{Re}(s)$ is sufficiently large and has a meromorphic continuation to the whole complex plane in $s$. By \cite{J09} and \cite[Proposition 3.2]{JLSa}, the normalized Rankin-Selberg integral
	\[
	\RZ^{\circ}(f,f',\varphi_s;\chi_s,\mathrm{d}g):=\frac{\RL(ns,\eta\chi^n)}{\RL(s,\pi_{\mu}\times\pi_{\nu}\times\chi)}\RZ(f,f',\varphi_s;\chi_s,\mathrm{d}g)
	\]
	is holomorphic at $s=0$ and specializing to $s=0$ yields a non-zero linear functional
	\[
	\RZ_{\mu,\nu,\chi}^{\circ}\in\mathrm{Hom}_{\GL_{n}(\K)}\left(\pi_{\mu}\widehat{\otimes}\pi_{\nu}\widehat{\otimes}I_{\eta,\chi},\mathfrak{M}_n^{\vee}\right).
	\]
	
	Define a vector space
	\[
	\begin{aligned}
		\RH_{\mu,\nu,\chi}:=\RH^{b_n}\left(\mathfrak{g}_n,\widetilde{K}_n;F_{\mu}^{\vee}\otimes\pi_{\mu}\right)&\otimes\RH^{b_{n}}\left(\mathfrak{g}_{n},\widetilde{K}_{n};F_{\nu}^{\vee}\otimes\pi_{\nu}\right)\\
		&\otimes\RH^{c_n}\left(\mathfrak{g}_{n},\widetilde{K}_{n};F_{\eta,\chi}^{\vee}\otimes I_{\eta,\chi}\right).
	\end{aligned}
	\]
	As in the $\GL_{n+1}\times\GL_{n}$ case, we define the archimedean modular symbol, which is a linear functional
	\begin{equation}
		\label{wpmunuchi}
		\wp_{\mu,\nu,\chi}:\RH_{\mu,\nu,\chi}\otimes\mathfrak{O}_{n,n}\to\C
	\end{equation}
	as the composition
	\[
	\begin{aligned}
		&\RH_{\mu,\nu,\chi}\otimes\mathfrak{O}_{n,n}\\
		\xrightarrow{\mathrm{res}}	& \RH^{n^2-1}\left(\mathfrak{g}_{n},\widetilde{K}_{n};F_{\mu}^{\vee}\otimes F_{\nu}^{\vee}\otimes F_{\eta,\chi}^{\vee}\otimes\pi_{\mu}\widehat{\otimes}\pi_{\nu}\widehat{\otimes}I_{\eta,\chi}\right)\otimes\mathfrak{O}_{n,n}\\
		\xrightarrow{\phi_{\mu,\nu,\chi}\otimes\RZ_{\mu,\nu,\chi}^{\circ}}&\RH^{n^2-1}\left(\mathfrak{g}_{n},\widetilde{K}_{n};\mathfrak{M}_{n}^{\vee}\right)\otimes\mathfrak{O}_{n,n}=\C.
	\end{aligned}
	\]
	
	\subsection{The main theorem}
	
	Let $\mathbf{o}_{n+1,n}\in\mathfrak{O}_{n+1,n}$ and $\mathbf{o}_{n,n}\in\mathfrak{O}_{n,n}$ be orientations fixed in Section \ref{orientation}. We fix generators
	\[
	[\kappa_{\mu}]\in\RH^{b_{n'}}\left(\mathfrak{g}_{n'},\widetilde{K}_{n'};F_{\mu}^{\vee}\otimes\pi_{\mu}\right),\quad
	[\kappa_{\nu}]\in\RH^{b_{n}}\left(\mathfrak{g}_{n},\widetilde{K}_{n};F_{\nu}^{\vee}\otimes\pi_{\nu}\right)
	\]
	in Section \ref{generic} and generators
	\[
	[\kappa_{\eta,\chi}]\in\RH^{c_n}\left(\mathfrak{g}_n,\widetilde{K}_n;F_{\eta,\chi}^{\vee}\otimes I_{\eta,\chi}\right)
	\]
	in Section \ref{induced}. Denote by
	\[
	[\kappa_{\chi}]\in\RH^0\left(\mathfrak{g}_{n},\widetilde{K}_{n};F_{\chi}^{\vee}\otimes(\chi\circ\det)\right)
	\]
	the canonical generator. Following \cite[(24)]{JLSa}, we introduce the following constants:
	\begin{equation}\label{constants}
		\begin{aligned}
			c_{\mu,\nu,\chi}&:=\prod_{i+k\leq n+1}\mathrm{i}^{\mu_i^{\iota}+\mu_i^{\overline{\iota}}+\nu_k^{\iota}+\nu_k^{\overline{\iota}}+\chi_{\iota}+\chi_{\overline{\iota}}},\\
			c'_{\mu,\nu,\chi}&:=\prod_{i+k\leq n}\mathrm{i}^{\mu_i^{\iota}+\mu_i^{\overline{\iota}}+\nu_k^{\iota}+\nu_k^{\overline{\iota}}+\chi_{\iota}+\chi_{\overline{\iota}}-1},\\
			\varepsilon_{\mu,\nu,\chi}&:=\prod_{i>k,\, i+k\leq n+1}(-1)^{\mu_i^{\iota}+\mu_i^{\overline{\iota}}+\nu_k^{\iota}+\nu_k^{\overline{\iota}}+\chi_{\iota}+\chi_{\overline{\iota}}},\\
			\varepsilon'_{\mu,\nu,\chi}&:=(-1)^{\frac{n(n-1)}{2}(\chi_{\iota}+\chi_{\overline{\iota}}-1)}\cdot\prod_{i=1}^n(-1)^{(n+1)\nu_i^{\overline{\iota}}}\\
            &\qquad\qquad\qquad\cdot\prod_{i>k,\,i+k\leq n}(-1)^{\mu_i^{\iota}+\nu_k^{\iota}+\mu_{n+1-i}^{\overline{\iota}}+\nu_{n+1-k}^{\overline{\iota}}+\chi_{\iota}+\chi_{\overline{\iota}}-1}.
		\end{aligned}
	\end{equation}
	Our main theorem is stated as follows. 

	\begin{thmp}
		\label{mainthm}
		Retain the notations and assumptions as above. We have that
		\begin{itemize}
			\item[(1)] When $n'=n+1$ we have that
			\[
			\begin{aligned}
				&\wp_{\mu,\nu,\chi}\left([\kappa_{\mu}]\otimes[\kappa_{\nu}]\otimes[\kappa_{\chi}]\otimes\mathbf{o}_{n+1,n}\right)\\
				=&\,2^{-n(n+1)}\cdot\mathrm{i}^{\frac{n(n-1)}{2}}\cdot c_{\mu,\nu,\chi}\cdot\varepsilon_{\mu,\nu,\chi}.
			\end{aligned}
			\]
			\item[(2)] When $n'=n$ we have that
			\[
			\begin{aligned}
				&\wp_{\mu,\nu,\chi}\left([\kappa_{\mu}]\otimes[\kappa_{\nu}]\otimes[\kappa_{\eta,\chi}]\otimes\mathbf{o}_{n,n}\right)\\
				=&\,2^{2-n(n+1)}\cdot (-1)^{\frac{n(n-1)}{2}}\cdot c'_{\mu,\nu,\chi}\cdot\varepsilon'_{\mu,\nu,\chi}.
			\end{aligned}
			\]
		\end{itemize}
	\end{thmp}

	We remark that the case $n'=n+1$ is essentially due to \cite[Theorem 2.17]{HMN}. But the result differs from the one there because of the different choice of generators $[\kappa_{\mu}]$ and $[\kappa_{\nu}]$.


	\section{Lie algebras, measures and orientations}
	
	\subsection{Measures and orientations}\label{orientation}
	
	Write $\mathfrak{g}_n=\mathfrak{k}_n\oplus\mathfrak{p}_n=\widetilde{\mathfrak{k}}_n\oplus\widetilde{\mathfrak{p}}_n$,  where $\mathfrak{p}_n$ is the orthogonal complement of $\mathfrak{k}_n$ in $\mathfrak{g}_n$ and $\widetilde{\mathfrak{p}}_n$ is the orthogonal complement of $[\mathfrak{g}_n,\mathfrak{g}_n]\cap\mathfrak{k}_n$ in $[\mathfrak{g}_n,\mathfrak{g}_n]$, with respect to the nondegenerate invariant trace form on $\mathfrak{g}_n$ and its restriction to $[\mathfrak{g}_n,\mathfrak{g}_n]$, respectively. Write $\C_{\iota}=\C$ (resp. $\C_{\overline{\iota}}=\C$), viewed as a $\K$-algebra via $\iota$ (resp. $\overline{\iota}$). Denote by $\mathfrak{gl}_n(\C_{\iota})$ (resp. $\mathfrak{gl}_n(\C_{\overline{\iota}})$) the Lie algebra of the Lie group $\GL_n(\C_{\iota})$ (resp. $\GL_n(\C_{\overline{\iota}})$) so that $\mathfrak{g}_n=\mathfrak{gl}_n(\C_{\iota})\oplus\mathfrak{gl}_n(\C_{\overline{\iota}})$. Clearly
	\[
	\{\varepsilon_{i,j}^{\iota}\}_{1\leq i,j\leq n},\qquad\text{resp. }\{\varepsilon_{i,j}^{\overline{\iota}}\}_{1\leq i,j\leq n}
	\]
	is a basis of $\mathfrak{gl}_n(\C_{\iota})$ (resp. $\mathfrak{gl}_n(\C_{\overline{\iota}})$), where $\varepsilon_{i,j}^{\iota}$ (resp. $\varepsilon_{i,j}^{\overline{\iota}}$) is the $n\times n$ matrix with $1$ at the $(i,j)$-th entry and $0$ elsewhere, viewed as elements of $\mathfrak{gl}_n(\C_{\iota})$ (resp. $\mathfrak{gl}_n(\C_{\overline{\iota}})$). Put
	\[
	\begin{aligned}
		e_{i,j}&:=\varepsilon_{i,j}^{\iota}+\varepsilon_{j,i}^{\overline{\iota}}, &\qquad & 1\leq i,j\leq n, \\
		e_{i}&:=\varepsilon_{i,i}^{\iota}+\varepsilon_{i,i}^{\overline{\iota}}-\varepsilon_{n,n}^{\iota}-\varepsilon_{n,n}^{\overline{\iota}},&\qquad& 1\leq i\leq n-1.
	\end{aligned}
	\]
	Then
	\[
	\{e_{i,j}\}_{1\leq i,j\leq n}\qquad\text{and}\qquad\{e_i\}_{1\leq i\leq n-1}\sqcup \{e_{i,j}\}_{1\leq i\neq j\leq n}
	\]
	are bases of $\mathfrak{p}_n$ and $\widetilde{\mathfrak{p}}_n$ respectively. Denote by 
	\[
	\{e^{\vee}_{i,j}\}_{1\leq i,j\leq n}\qquad\text{and}\qquad\{e^{\vee}_i\}_{1\leq i\leq n-1}\sqcup \{e^{\vee}_{i,j}\}_{1\leq i\neq j\leq n}
	\]
	the dual bases of $\mathfrak{p}^{\vee}_n$ and $\widetilde{\mathfrak{p}}^{\vee}_n$ with respect to the above bases of $\mathfrak{p}_n$ and $\widetilde{\mathfrak{p}}_n$ respectively.
	
	We define elements $\mathfrak{o}_{n+1,n}\in\wedge^{n^2}\mathfrak{p}_n^{\vee}$ and $\mathfrak{o}_{n,n}\in\wedge^{n^2-1}\widetilde{\mathfrak{p}}_n^{\vee}$ by
	\[
	\begin{aligned}
		\mathfrak{o}_{n+1,n}=&\left(\bigwedge_{1\leq i\leq n}\bigwedge_{1\leq j\leq i}e^{\vee}_{i,j}\right)\wedge\left(\bigwedge_{\substack{2\leq j\leq n}}\bigwedge_{1\leq i\leq j-1}e^{\vee}_{i,j}\right) \\
		=&e^{\vee}_{1,1}\wedge e^{\vee}_{2,1}\wedge e^{\vee}_{2,2}\wedge\cdots\wedge e^{\vee}_{n,1}\wedge e^{\vee}_{n,2}\wedge\cdots\wedge e^{\vee}_{n,n}\\
		&\wedge e^{\vee}_{1,2}\wedge e^{\vee}_{1,3}\wedge e^{\vee}_{2,3}\wedge\cdots\wedge e^{\vee}_{1,n}\wedge e^{\vee}_{2,n}\wedge\cdots\wedge e^{\vee}_{n-1,n}
	\end{aligned}
	\]
	and
	\[
	\begin{aligned}
		\mathfrak{o}_{n,n}=&\left(\bigwedge_{2\leq i\leq n}\bigwedge_{1\leq j\leq i-1}e^{\vee}_{i,j}\right)\wedge\left(\bigwedge_{\substack{2\leq j\leq n}}\bigwedge_{1\leq i\leq j-1}e^{\vee}_{i,j}\right)\wedge\left(\bigwedge_{1\leq i\leq n-1}e_{i}^{\vee}\right)\\
		=&e_{2,1}^{\vee}\wedge e_{3,1}^{\vee}\wedge e_{3,2}^{\vee}\wedge\cdots\wedge e_{n,1}^{\vee}\wedge e_{n,2}^{\vee}\wedge\cdots\wedge e_{n,n-1}^{\vee}\\
		&\wedge e^{\vee}_{1,2}\wedge e^{\vee}_{1,3}\wedge e^{\vee}_{2,3}\wedge\cdots\wedge e^{\vee}_{1,n}\wedge e^{\vee}_{2,n}\wedge\cdots\wedge e^{\vee}_{n-1,n}\\
		&\wedge e_{1}^{\vee}\wedge e_{2}^{\vee}\wedge\cdots\wedge e_{n-1}^{\vee}.
	\end{aligned}
	\]
	Following the calculation in \cite[Section 3.2]{HMN}, $\mathfrak{o}_{n+1,n}$ (resp. $\mathfrak{o}_{n,n}$) determines an orientation $\mathbf{o}_{n+1,n}\in\mathfrak{O}_{n+1,n}$ (resp. $\mathbf{o}_{n,n}\in\mathfrak{O}_{n,n}$) of the manifold $\GL_n(\K)/K_n$ (resp. $\GL_n(\K)/\widetilde{K}_n$). 
	
	We fix a distinguished measure $\mathrm{d}^{\circ}g\in\mathfrak{M}_{n}$ to be 
	\[
	\mathrm{d}^{\circ}g:=\frac{\Gamma_{\C}(1)\Gamma_{\C}(2)\cdots\Gamma_{\C}(n)}{|\det g|_{\K}^n}\cdot\prod_{1\leq i,j\leq n}\mathrm{d}g_{i,j},\quad g=[g_{i,j}]_{1\leq i,j\leq n}\in\GL_{n}(\K)
	\]
	where each $\mathrm{d}g_{i,j}$ is the self-dual Haar measure with respect to $\psi_{\K}$ and $\Gamma_{\C}(s)=2(2\pi)^{-s}\Gamma(s)$ with $\Gamma(s)$ the standard gamma function. One checks that $\mathrm{d}^{\circ}g$ is the same as the measure chosen in \cite[Section 2.2]{IM}. By abusing notation, we also write $\mathrm{d}^{\circ}g$ for the quotient measure on both $\GL_n(\K)/K_n$ and $\GL_n(\K)/\widetilde{K}_n$. Recall that by identification \eqref{orientationid}, $\mathfrak{o}_{n+1,n}\otimes\mathbf{o}_{n+1,n}$ (resp. $\mathfrak{o}_{n,n}\otimes\mathbf{o}_{n,n}$) defines an invariant measure on $\GL_n(\K)/K_n$ (resp. $\GL_n(\K)/\widetilde{K}_n$). It is straightforward to compute that
	\begin{equation}
		\label{measureid}
		\begin{aligned}
			\mathfrak{o}_{n+1,n}\otimes\mathbf{o}_{n+1,n}&=2^{-n(n+1)}\cdot\mathrm{i}^{-b_n}\cdot\mathrm{d}^{\circ}g,&\qquad&n'=n+1,\\
			\mathfrak{o}_{n,n}\otimes\mathbf{o}_{n,n}&=\frac{4}{n}\cdot2^{-n(n+1)}\cdot\mathrm{i}^{-b_n}\cdot\mathrm{d}^{\circ}g,&\qquad&n'=n.
		\end{aligned}
	\end{equation}
	The detailed computation for the case $n'=n+1$ is provided in \cite[Section 3.2]{HMN} and the case $n'=n$ is similar.
	
	\subsection{The $K_n$-types $\Xi_n$ and $\Xi_n'$}
	
	For $n\geq 2$, recall that $\widetilde{\mathfrak{p}}_n$ is an irreducible $K_n$-representation of highest weight 
	\[
	(1,0,\dots,0,-1)
	\]
	with the adjoint action of $K_n$. Clearly, $\widetilde{\mathfrak{p}}_1=0$ when $n=1$. Let $\Xi_n$ be an irreducible $K_n$-representation of highest weight 
	\[
	(n-1,n-3,\dots,1-n),
	\]
	and $\Xi_n'$ an irreducible $K_n$-representation of highest weight 
	\[
	(n-1,-1,\dots,-1).
	\]
	Then $\Xi^{\vee}_n$ occurs with multiplicity one in $\wedge^{b_n}\widetilde{\mathfrak{p}}^{\vee}_n$ by \cite[Lemma 4.5]{DX}, and $\Xi_n'^{\vee}$ occurs with multiplicity one in $\wedge^{c_n}\widetilde{\mathfrak{p}}^{\vee}_n$ by \cite[Lemma 4.3]{DX}. Moreover,
	\[
	\omega_n:=\bigwedge_{2\leq i\leq n}\bigwedge_{1\leq j\leq i-1}e_{i,j}^{\vee}=e_{2,1}^{\vee}\wedge e_{3,1}^{\vee}\wedge e_{3,2}^{\vee}\wedge\cdots\wedge e_{n,1}^{\vee}\wedge e_{n,2}^{\vee}\wedge\cdots\wedge e_{n,n-1}^{\vee}
	\]
	is a highest weight vector of the $K_n$-type $\Xi^{\vee}_n$ in $\wedge^{b_n}\widetilde{\mathfrak{p}}^{\vee}_n$, and
	\[
	\omega_n':=\bigwedge_{1\leq j\leq n-1}e_{n,j}^{\vee}=e_{n,1}^{\vee}\wedge e_{n,2}^{\vee}\wedge\cdots\wedge e_{n,n-1}^{\vee}
	\]
	is a highest weight vector of the $K_n$-type $\Xi_n'^{\vee}$ in $\wedge^{c_n}\widetilde{\mathfrak{p}}^{\vee}_n$.

	\subsection{The Gelfand-Tsetlin basis}
	
	For a dominant weight 
	\[
	\lambda=(\lambda_1,\lambda_2,\dots,\lambda_n)\in\Z^n,\qquad\lambda_1\geq\lambda_2\geq\cdots\geq\lambda_n,
	\]
	denote by $\Xi(\lambda)$ an irreducible $K_n$-representation of highest weight $\lambda$. Let $\RG(\lambda)$ be the set of integral triangular arrays
	\[
	M=(m_{i,j})_{1\leq i\leq j\leq n},\qquad m_{i,j}\in\Z
	\]
	such that $m_{i,n}=\lambda_i$ for $1\leq i\leq n$ and $m_{j,k}\geq m_{j,k-1}\geq m_{j+1,k}$ for $1\leq j<k\leq n$. We identify $\Xi(\lambda)^{\vee}=\Xi(\lambda^{\vee})$, with
	\[
	\lambda^{\vee}:=(-\lambda_n,-\lambda_{n-1},\dots,-\lambda_1).
	\]
	For every $M\in\RG(\lambda)$, denote 
	\[
	M^{\vee}=(m^{\vee}_{i,j})_{1\leq i\leq j\leq n}\qquad\text{with }m_{i,j}^{\vee}:=-m_{j+1-i,j}
	\]
	so that $M^{\vee}\in\RG(\lambda^{\vee})$. By Weyl's unitary trick \cite[Proposition 7.15]{Kn02}, we can shift freely between $\RU(n)$-representations and holomorphic $\GL_n(\C)$-representations. 
	
	We fix a $\RU(n)$-invariant hermitian inner product $\langle\cdot,\cdot\rangle$ on $\Xi(\lambda)$. Then $\Xi(\lambda)$ has a (rational) Gelfand-Tsetlin basis 
	\[
	\{\xi_M\}_{M\in\RG(\lambda)}
	\]
	such that for every $M,N\in\RG(\lambda)$ it holds that
	\begin{equation}\label{product}
	\langle\xi_M,\xi_N\rangle=\begin{cases}
		\mathrm{r}(M),& \text{if }M=N,\\
		0,&\text{otherwise},
	\end{cases}
	\end{equation}
	where
	\begin{equation}\label{r}
		\begin{aligned}
			\mathrm{r}(M)&=\prod_{1\leq i\leq j<k\leq n}\frac{(m_{i,k}-m_{j,k-1}-i+j)!(m_{i,k-1}-m_{j+1,k}-i+j)!}{(m_{i,k-1}-m_{j,k-1}-i+j)!(m_{i,k}-m_{j+1,k}-i+j)!}.
		\end{aligned}
	\end{equation}
	See \cite[Section 2.5]{IM} (and also \cite{GT50, Z85}) for the definition of $\xi_M$, which will not be recalled here. Let
	\[
	H(\lambda)=(h_{i,j})_{1\leq i\leq j\leq n}\in\RG(\lambda)\qquad\text{with }h_{i,j}=\lambda_i.
	\]
	Then $\xi_{H(\lambda)}$ is a highest weight vector of $\Xi(\lambda)$. 
	
	We apply the above discussion to the representations $\Xi_n$, $\Xi_n'$ and their contragredients. Write
	\[
	\begin{aligned}
		\RG_n&:=\RG\left((n-1,n-3,\dots,1-n)\right),\\
		\RG_n'&:=\RG\left((n-1,-1,\dots,-1)\right),\quad\check{\RG}'_n:=\RG\left((1,\dots,1,1-n)\right),
	\end{aligned}
	\]
	and
	\[
	\begin{aligned}
		H_n&:=H\left((n-1,n-3,\dots,1-n)\right)\in\RG_n,\\
		H_n'&:=H\left((n-1,-1,\dots,-1)\right)\in\RG'_n,\, \check{H}_n':=H\left((1,\dots,1,1-n)\right)\in\check{\RG}_n'.
	\end{aligned}
	\]
	We fix an inclusion $\RI_n:\Xi^{\vee}_n\hookrightarrow\wedge^{b_n}\widetilde{\mathfrak{p}}_n^{\vee}$ such that $\RI_n(\xi_{H_n})=\omega_n$ and denote $\omega_{M}:=\RI_n(\xi_{M})$ for each $M\in\RG_n$. Similarly, we fix an inclusion $\RI_n':\Xi'^{\vee}_n\hookrightarrow\wedge^{c_n}\widetilde{\mathfrak{p}}_n^{\vee}$ such that $\RI'_n(\xi_{\check{H}'_n})=\omega'_n$ and denote $\omega'_{M}:=\RI'_n(\xi_{M})$ for each $M\in\check{\RG}'_n$.

	\section{Generators for cohomology spaces}
	
	\subsection{Whittaker functionals}
	\label{whittaker}
	
	Let $\delta=(\delta_1,\dots,\delta_n)\in\Z^n$ and $\lambda=(\lambda_1,\dots,\lambda_n)\in\C^n$. Consider the principal series representation
	\[
	I(\delta;\lambda):=\mathrm{Ind}^{\GL_n(\K)}_{\overline{\RB}_n(\K)}\left(\iota^{\frac{\lambda_1+\delta_1}{2}}\overline{\iota}^{\frac{\lambda_1-\delta_1}{2}}\otimes\cdots\otimes\iota^{\frac{\lambda_n+\delta_n}{2}}\overline{\iota}^{\frac{\lambda_n-\delta_n}{2}}\right).
	\]
	For example,
	\[
	\pi_{\nu}= I\left(\widetilde{\nu}_1^{\iota}-\widetilde{\nu}_n^{\overline{\iota}},\dots,\widetilde{\nu}_n^{\iota}-\widetilde{\nu}_1^{\overline{\iota}};\widetilde{\nu}_1^{\iota}+\widetilde{\nu}_n^{\overline{\iota}},\dots,\widetilde{\nu}_n^{\iota}+\widetilde{\nu}_1^{\overline{\iota}}\right).
	\]
	Recall the Jacquet integral 
	\[
	\CJ_{\psi_{n,\K}}:I(\delta;\lambda)\to\C,\qquad f\mapsto\int_{\RN_n(\K)}f(x)\psi_{n,\K}^{-1}(x)\mathrm{d}x,
	\]
	which converges absolutely when
	\[
	\Re(\lambda_1)<\Re(\lambda_2)<\cdots<\Re(\lambda_n).
	\]
	By means of the holomorphic continuation (\cite[Theorem 15.4.1]{Wa2}), we obtain a Whittaker functional
	\[
	\CJ_{\psi_{n,\K}}\in\mathrm{Hom}_{\RN_n(\K)}\left(I(\delta;\lambda),\psi_{n,\K}\right).
	\]
	
	Following \cite[Section 2.4.1, Section 5.2]{HMN}, we normalize the Whittaker functionals 
	\[
	\lambda_{\mu}\in\mathrm{Hom}_{\RN_{n'}(\K)}(\pi_{\mu},\psi_{n',\K}),\qquad\lambda_{\nu}\in\mathrm{Hom}_{\RN_{n}(\K)}(\pi_{\nu},\overline{\psi_{n,\K}})
	\]
	by
	\[
	\begin{aligned}
		\lambda_{\mu}:=\mathrm{i}^{d_{\mu}}\cdot\Gamma_{\C}(\mu)\cdot\CJ_{\psi_{n',\K}},\qquad \lambda_{\nu}:=\mathrm{i}^{d_{\nu}}\cdot\Gamma_{\C}(\nu)\cdot\CJ_{\overline{\psi_{n,\K}}},
	\end{aligned}
	\]
	where
	\[
	\begin{aligned}
		d_{\mu}&=\sum_{i=1}^{n'}(n'-i)(\widetilde{\mu}_i^{\iota}-\widetilde{\mu}_{n'+1-i}^{\overline{\iota}}),\\
		\Gamma_{\C}(\mu)&=\prod_{1\leq i<j\leq n'}\Gamma_{\C}\left(\max\left\{\widetilde{\mu}_j^{\iota}-\widetilde{\mu}_i^{\iota},\widetilde{\mu}_{n'+1-j}^{\overline{\iota}}-\widetilde{\mu}_{n'+1-i}^{\overline{\iota}}\right\}+1\right), 
	\end{aligned}
	\]
	and $d_{\nu}$, $\Gamma_{\C}(\nu)$ are defined by the same formulas respectively.

	\subsection{Generic cohomological representations}\label{generic}
	
	For every positive integer $n$, we write $0_n$ for the zero element of the abelian group $\Z^n\times\Z^n$. By \cite[Lemma 4.2]{DX}, the $K_n$-representation $\Xi_n$ occurs with multiplicity one in $\pi_{0_n}$. We fix an inclusion $\RJ_n:\Xi_n\hookrightarrow\pi_{0_n}$ such that $\RJ_n(\xi_{H_n})(1_n)=1$ and denote $f_M:=\RJ_n(\xi_M)$ for each $M\in\RG_n$. By \cite[Proposition 9.4.3]{Wa1}, we have that
	\begin{equation}\label{W1}
		\RH^{b_n}\left(\mathfrak{g}_n,\widetilde{K}_n;\pi_{0_n}\right)=\left(\wedge^{b_n}\widetilde{\mathfrak{p}}^{\vee}_n\otimes\pi_{0_n}\right)^{K_n}=\left(\RI_n(\Xi^{\vee}_n)\otimes\RJ_n(\Xi_n)\right)^{K_n}.
	\end{equation}
	Following \cite[(2.61)]{HMN}, we fix a generator
	\[
	[\kappa_{0_n}]=\sum_{M\in\RG_n}\frac{(-1)^{\mathrm{q}(M)}}{\mathrm{r}(M)}\omega_{M^{\vee}}\otimes f_{M}\in\RH^{b_n}\left(\mathfrak{g}_n,\widetilde{K}_n;\pi_{0_n}\right),
	\]
	where
	\begin{equation}\label{q}
		\mathrm{q}(M):=\sum_{1\leq i\leq j\leq n-1}m_{i,j}.
	\end{equation}

	With respect to the fixed Whittaker functionals
	\[
	\begin{aligned}
		\lambda_{\mu}&\in\mathrm{Hom}_{\RN_{n'}(\K)}(\pi_{\mu},\psi_{n',\K}),&\qquad\lambda_{\nu}&\in\mathrm{Hom}_{\RN_{n}(\K)}(\pi_{\nu},\overline{\psi_{n,\K}}),\\
		\lambda_{0_{n'}}&\in\mathrm{Hom}_{\RN_{n'}(\K)}(\pi_{0_{n'}},\psi_{n',\K}),&\qquad\lambda_{0_n}&\in\mathrm{Hom}_{\RN_{n}(\K)}(\pi_{0_n},\overline{\psi_{n,\K}}),
	\end{aligned}
	\]
	we have translation maps (\cite[Proposition 1.1]{JLSa})
	\[
	\jmath_{\mu}\in\mathrm{Hom}_{\GL_{n'}(\K)}(\pi_{0_{n'}},F_{\mu}^{\vee}\otimes\pi_{\mu}),\qquad
	\jmath_{\nu}\in\mathrm{Hom}_{\GL_n(\K)}(\pi_{0_{n}},F_{\nu}^{\vee}\otimes\pi_{\nu}),
	\]
	which induce linear isomorphisms
	\[
	\begin{aligned}
		\jmath_{\mu}:&\RH^{b_{n'}}\left(\mathfrak{g}_{n'},\widetilde{K}_{n'};\pi_{0_{n'}}\right)\xrightarrow{\sim}\RH^{b_{n'}}\left(\mathfrak{g}_{n'},\widetilde{K}_{n'};F_{\mu}^{\vee}\otimes\pi_{\mu}\right),\\
		\jmath_{\nu}:&\RH^{b_n}\left(\mathfrak{g}_n,\widetilde{K}_n;\pi_{0_n}\right)\xrightarrow{\sim}\RH^{b_n}\left(\mathfrak{g}_n,\widetilde{K}_n;F_{\nu}^{\vee}\otimes\pi_{\nu}\right).
	\end{aligned}
	\]
	We then define
	\[
	\begin{aligned}
		[\kappa_{\mu}]&:=\jmath_{\mu}([\kappa_{0_{n'}}])\in\RH^{b_{n'}}\left(\mathfrak{g}_{n'},\widetilde{K}_{n'};F_{\mu}^{\vee}\otimes\pi_{\mu}\right),\\
		[\kappa_{\nu}]&:=\jmath_{\nu}([\kappa_{0_{n}}])\in\RH^{b_{n}}\left(\mathfrak{g}_{n},\widetilde{K}_{n};F_{\nu}^{\vee}\otimes\pi_{\nu}\right).
	\end{aligned}
	\]
	
	\subsection{Degenerate principal series representations}\label{induced}
	
	Denote by $\eta_0$ the trivial character and $\chi_0:=\overline{\iota}|_{\K^{\times}}$. By \cite[Lemma 4.1]{DX}, the $K_n$-representation $\Xi_n'$ occurs with multiplicity one in $I_{\eta_0,\chi_0}$. Denote by $w_n\in\GL_n(\Z)$ the matrix with ones on the anti-diagonal and zeros elsewhere. We fix an inclusion $\RJ_n':\Xi_n'\hookrightarrow I_{\eta_0,\chi_0}$ such that $\RJ_n'(\xi_{H'_n})(w_n)=1$ and denote $\varphi_M:=\RJ_n'(\xi_{M})$ for each $M\in\RG_n'$. 
	
	Note that $I_{\eta_0,\chi_0}$ is unitarizable. Hence by \cite[Proposition 9.4.3]{Wa1}, we have that
	\begin{equation}\label{W2}
		\RH^{c_n}\left(\mathfrak{g}_n,\widetilde{K}_n;I_{\eta_0,\chi_0}\right)=\left(\wedge^{c_n}\widetilde{\mathfrak{p}}^{\vee}_n\otimes I_{\eta_0,\chi_0}\right)^{K_n}=\left(\RI_n'(\Xi'^{\vee}_n)\otimes\RJ_n'(\Xi'_n)\right)^{K_n}.
	\end{equation}
	We fix a generator
	\[
	[\kappa_{\eta_0,\chi_0}]=\sum_{M\in\RG_n'}\frac{(-1)^{\mathrm{q}(M)}}{\mathrm{r}(M)}\omega'_{M^{\vee}}\otimes\varphi_M\in\RH^{c_n}\left(\mathfrak{g}_n,\widetilde{K}_n;I_{\eta_0,\chi_0}\right).
	\]
	
	By \cite[Proposition 1.2]{JLSa}, we have a translation map
	\[
	\jmath_{\eta,\chi}\in\mathrm{Hom}_{\GL_n(\K)}(I_{\eta_0,\chi_0},F_{\eta,\chi}^{\vee}\otimes I_{\eta,\chi}),
	\]
	which induces a linear isomorphism
	\[
	\begin{aligned}
		\jmath_{\eta,\chi}:\RH^{c_n}\left(\mathfrak{g}_n,\widetilde{K}_n;I_{\eta_0,\chi_0}\right)\xrightarrow{\sim}\RH^{c_n}\left(\mathfrak{g}_n,\widetilde{K}_n;F_{\eta,\chi}^{\vee}\otimes I_{\eta,\chi}\right).
	\end{aligned}
	\]
	We then define
	\[
	[\kappa_{\eta,\chi}]:=\jmath_{\eta,\chi}([\kappa_{\eta_0,\chi_0}])\in\RH^{c_n}\left(\mathfrak{g}_n,\widetilde{K}_n;F_{\eta,\chi}^{\vee}\otimes I_{\eta,\chi}\right).
	\]
	
	\section{Computation of archimedean modular symbols}
	
	\subsection{The $\GL(n)\times\GL(n+1)$ case}
	\label{n+1}
	
	We sketch the proof of Theorem \ref{mainthm}(1) following \cite{HMN} for completeness. In view of the archimedean period relation \cite[Theorem 1.6]{JLSa} and our choice of generators in Section \ref{generic}, it suffices to compute
	\[
	\wp_{0_{n+1},0_n,\mathbf{1}}\left([\kappa_{0_{n+1}}]\otimes[\kappa_{0_n}]\otimes[\kappa_{\mathbf{1}}]\otimes\mathbf{o}_{n+1,n}\right),
	\]
	where $\mathbf{1}$ is the trivial character of $\K^{\times}$.
	
	Note that the embedding $\GL_n\hookrightarrow\GL_{n+1}$ induces an inclusion $\mathfrak{p}_n\hookrightarrow\mathfrak{p}_{n+1}$, which further induces a restriction map
	\[
	\wedge^{b_{n+1}}\mathfrak{p}_{n+1}^{\vee}\to\wedge^{b_{n+1}}\mathfrak{p}_{n}^{\vee}.
	\]
	Taking the composition of the above restriction map with the usual wedge product we obtain a map
	\[
	\begin{aligned}
		\wedge(\cdot,\cdot):\wedge^{b_{n+1}}\mathfrak{p}_{n+1}^{\vee}\otimes\wedge^{b_n}\mathfrak{p}_n^{\vee}&\to\wedge^{n^2}\mathfrak{p}_n^{\vee},\\
		\omega\otimes\omega'&\mapsto\omega\wedge\omega'.
	\end{aligned}
	\]
	In view of \eqref{W1}, this induces the map
	\[
	\begin{aligned}	&\,\RH^{b_{n+1}}(\mathfrak{g}_{n+1},\widetilde{K}_{n+1};\pi_{0_{n+1}})\otimes\RH^{b_{n}}(\mathfrak{g}_{n},\widetilde{K}_{n};\pi_{0_{n}})\\
		=&\,\left(\wedge^{b_{n+1}}\widetilde{\mathfrak{p}}_{n+1}^{\vee}\otimes\pi_{0_{n+1}}\right)^{K_{n+1}}\otimes\left(\wedge^{b_{n}}\widetilde{\mathfrak{p}}_{n}^{\vee}\otimes\pi_{0_{n}}\right)^{K_{n}}\\
		\xrightarrow{\wedge(\cdot,\cdot)}&\,\left(\wedge^{n^2}\mathfrak{p}_n^{\vee}\otimes\pi_{0_{n+1}}\otimes\pi_{0_n}\right)^{K_n},
	\end{aligned}
	\]
	which will also be denoted by $\wedge(\cdot,\cdot)$. 
	
	Since $\wedge^{n^2}\mathfrak{p}_n^{\vee}$ is generated by the $K_n$-invariant vector $\mathfrak{o}_{n+1,n}$, by \cite[Lemma 4.3]{IM}, there exists $C_{n+1,n}\in\C$ such that
	\[
	\wedge([\kappa_{0_{n+1}}]\otimes[\kappa_{0_n}])=C_{n+1,n}\cdot\mathfrak{o}_{n+1,n}\otimes\left(\sum_{M\in\RG_n}\frac{(-1)^{\mathrm{q}(M)}}{\mathrm{r}(M)}f_{\widetilde{M}}\otimes f_{M^{\vee}}\right),
	\]
	where $\widetilde{M}=(\widetilde{m}_{i,j})_{1\leq i\leq j\leq n+1}\in\RG_{n+1}$ is given by $\widetilde{m}_{i,n+1}=n+2-2i$ for $1\leq i\leq n+1$ and $\widetilde{m}_{i,j}=m_{i,j}$ for $1\leq i\leq j\leq n$. A direct computation following the proof of \cite[Proposition 2.18]{HMN} gives
	\[
	C_{n+1,n}=(-1)^{\frac{n(n+1)}{2}+\frac{n(n-1)(n-2)}{6}}.
	\]
	Then by \eqref{measureid} and the definition of the archimedean modular symbol, we have that
	\[
	\begin{aligned}
		&\wp_{0_{n+1},0_n,\mathbf{1}}\left([\kappa_{0_{n+1}}]\otimes[\kappa_{0_n}]\otimes[\kappa_{\mathbf{1}}]\otimes\mathbf{o}_{n+1,n}\right)=2^{-n(n+1)}\cdot\mathrm{i}^{-b_n}\\
		&\qquad\cdot(-1)^{\frac{n(n+1)}{2}+\frac{n(n-1)(n-2)}{6}}\cdot\left(\sum_{M\in\RG_n}\frac{(-1)^{\mathrm{q}(M)}}{\mathrm{r}(M)}\RZ^{\circ}_{0_{n+1},0_n,\mathbf{1}}\left(f_{\widetilde{M}},f_{M^{\vee}};\mathbf{1},\mathrm{d}^{\circ}g\right)\right).
	\end{aligned}
	\]
	
	Using the result of \cite[Corollary 2.10]{IM}, one verifies that
	\[
	\sum_{M\in\RG_n}\frac{(-1)^{\mathrm{q}(M)}}{\mathrm{r}(M)}\RZ^{\circ}_{0_{n+1},0_n,\mathbf{1}}\left(f_{\widetilde{M}},f_{M^{\vee}};\mathbf{1},\mathrm{d}^{\circ}g\right)=(-1)^{\frac{n(n+1)(n+2)}{6}}.
	\]
	Therefore,
	\[
	\begin{aligned}
		&\wp_{0_{n+1},0_n,\mathbf{1}}\left([\kappa_{0_{n+1}}]\otimes[\kappa_{0_n}]\otimes[\kappa_{\mathbf{1}}]\otimes\mathbf{o}_{n+1,n}\right)\\
		=&\,2^{-n(n+1)}\cdot\mathrm{i}^{\frac{n(n-1)}{2}},
	\end{aligned}
	\]
	which completes the proof of Theorem \ref{mainthm}(1).
	
	\subsection{Reduction of the archimedean modular symbol}\label{sec:4.2}
	
	From now on we assume that $n'=n$. The rest of the paper is devoted to the proof of Theorem \ref{mainthm}(2) using the same strategy as in Section \ref{n+1}. Due to the archimedean period relation \cite[Theorem 1.6]{JLSa} and our choice of generators in Sections \ref{generic} and \ref{induced}, it suffices to compute
	\[
	\wp_{0_n,0_n,\chi_0}\left([\kappa_{0_n}]\otimes[\kappa_{0_n}]\otimes[\kappa_{\eta_0,\chi_0}]\otimes\mathbf{o}_{n,n}\right).
	\]
	
	In view of \eqref{W1} and \eqref{W2}, the usual wedge product
	\[
	\begin{aligned}
		\wedge(\cdot,\cdot,\cdot):\wedge^{b_n}\widetilde{\mathfrak{p}}_n^{\vee}\otimes\wedge^{b_n}\widetilde{\mathfrak{p}}_n^{\vee}\otimes\wedge^{c_n}\widetilde{\mathfrak{p}}_n^{\vee}&\to\wedge^{n^2-1}\widetilde{\mathfrak{p}}_n^{\vee},\\
		\omega\otimes\omega'\otimes\omega''&\mapsto\omega\wedge\omega'\wedge\omega''
	\end{aligned}
	\]
	induces a map
	\[
	\begin{aligned}
		&\RH^{b_{n}}(\mathfrak{g}_{n},\widetilde{K}_{n};\pi_{0_{n}})\otimes\RH^{b_{n}}(\mathfrak{g}_{n},\widetilde{K}_{n};\pi_{0_{n}})\otimes\RH^{c_{n}}(\mathfrak{g}_{n},\widetilde{K}_{n};I_{\eta_0,\chi_0})\\
		=&\left(\wedge^{b_n}\widetilde{\mathfrak{p}}_n^{\vee}\otimes\pi_{0_n}\right)^{K_n}\otimes\left(\wedge^{b_n}\widetilde{\mathfrak{p}}_n^{\vee}\otimes\pi_{0_n}\right)^{K_n}\otimes\left(\wedge^{c_n}\widetilde{\mathfrak{p}}_n^{\vee}\otimes I_{\eta_0,\chi_0}\right)^{K_n}\\
		\xrightarrow{\wedge(\cdot,\cdot,\cdot)}&\left(\wedge^{n^2-1}\widetilde{\mathfrak{p}}_n^{\vee}\otimes\pi_{0_n}\otimes\pi_{0_n}\otimes I_{\eta_0,\chi_0}\right)^{K_n}
	\end{aligned}
	\]
	which will also be denoted by $\wedge(\cdot,\cdot,\cdot)$.
	
	Note that $\wedge^{n^2-1}\widetilde{\mathfrak{p}}_n^{\vee}$ is generated by the $K_n$-invariant vector $\mathfrak{o}_{n,n}$. By \cite[Lemma 4.4]{IM}, there exists a constant $C_{n,n}\in\C$ such that
	\begin{equation}\label{cnn}
		\begin{aligned}
			\wedge\left([\kappa_{0_n}]\otimes\right.&\left.[\kappa_{0_n}]\otimes[\kappa_{\eta_0,\chi_0}]\right)=C_{n,n}\cdot\mathfrak{o}_{n,n}\\
			\otimes&\left(\sum_{M\in\RG_n}\sum_{M'\in\RG_n}\sum_{P\in\RG_n'}\frac{(-1)^{\mathrm{q}(M')}}{\mathrm{r}(M')}\mathrm{c}_{M'}^{M-1,P+1}\cdot f_M\otimes f_{M'^{\vee}}\otimes\varphi_{P}\right),
		\end{aligned}
	\end{equation}
	where the coefficients $\mathrm{c}_{M'}^{M-1,P+1}$ are rational numbers explicitly determined by \cite[Proposition 2.13]{IM}. Here we write $M-1:=(m_{i,j}-1)_{1\leq i\leq j\leq n}$ and $P+1:=(p_{i,j}+1)_{1\leq i\leq j\leq n}$ for $M=(m_{i,j})_{1\leq i\leq j\leq n}\in\RG_n$ and $P=(p_{i,j})_{1\leq i\leq j\leq n}\in\RG_n'$ respectively. Then by \eqref{measureid} and the definition of the archimedean modular symbol, we have that
	\begin{equation}\label{reduction}
		\begin{aligned}
			&\wp_{0_n,0_n,\chi_0}\left([\kappa_{0_n}]\otimes[\kappa_{0_n}]\otimes[\kappa_{\eta_0,\chi_0}]\otimes\mathbf{o}_{n,n}\right)\\
			=&\,\frac{4}{n}\cdot2^{-n(n+1)}\cdot\mathrm{i}^{-b_n}\cdot C_{n,n}\cdot\\
			&\left(\sum_{M\in\RG_n}\sum_{M'\in\RG_n}\sum_{P\in\RG_n'}\frac{(-1)^{\mathrm{q}(M')}}{\mathrm{r}(M')}\mathrm{c}_{M'}^{M-1,P+1}\cdot\RZ^{\circ}_{0_n,0_n,\chi_0}\left(f_M,f_{M'^{\vee}},\varphi_P;\chi_0,\mathrm{d}^{\circ}g\right)\right).
		\end{aligned}
	\end{equation}

	\subsection{Calculation of the constant $C_{n,n}$}\label{sec:4.3}
	
	In this subsection, we calculate the constant $C_{n,n}$ and show that
	\begin{equation}
		\label{CNN}
		C_{n,n}=\frac{(-1)^{\frac{n(n-1)(n-2)}{6}}}{(n-1)!\cdot\RC^{\circ}}
	\end{equation}
	where
	\[
	\RC^{\circ}=\prod_{1\leq i<j\leq n}\frac{3j-3i-1}{3j-3i+1}.
	\]
	
	Let $P_{\circ}=(p_{i,j})_{1\leq i\leq j\leq n}\in\RG_n'$ such that $p_{1,j}=j-1$ for $1\leq j\leq n$ and $p_{i,j}=-1$ for $2\leq i\leq j\leq n$. We calculate the constant $C_{n,n}$ by comparing the coefficients of the term
	\[
	f_{H_n}\otimes f_{H_n^{\vee}}\otimes \varphi_{P_{\circ}}
	\]
	on both sides of \eqref{cnn}. 
	
	\begin{lemp}We have that
		\label{lemA}
		\[
		(-1)^{\frac{n(n-1)(n+1)}{6}}\cdot(\RC^{\circ})^{-1}\cdot	\omega_{H^{\vee}_n}\wedge\omega_{H_n}\wedge\omega'_{P^{\vee}_{\circ}}=C_{n,n}\cdot\mathfrak{o}_{n,n}.
		\]
	\end{lemp}
	
	\begin{proof}
		On the left hand side of \eqref{cnn} we have the term
		\[
		\frac{(-1)^{\mathrm{q}(H_n)+\mathrm{q}(H_n^{\vee})+\mathrm{q}(P_{\circ})}}{\mathrm{r}(H_n)\mathrm{r}(H_n^{\vee})\mathrm{r}(P_{\circ})}\cdot(\omega_{H^{\vee}_n}\wedge\omega_{H_n}\wedge\omega'_{P^{\vee}_{\circ}})\otimes(f_{H_n}\otimes f_{H_n^{\vee}}\otimes \varphi_{P_{\circ}}),
		\]
		while on the right hand side of \eqref{cnn} we have 
		\[
		C_{n,n}\cdot\frac{(-1)^{\mathrm{q}(H_n)}}{\mathrm{r}(H_n)}\mathrm{c}_{H_n}^{H_n -1,P_{\circ}+1}\cdot\mathfrak{o}_{n,n}\otimes(f_{H_n}\otimes f_{H_n^{\vee}}\otimes \varphi_{P_{\circ}}).
		\]
		Therefore,
		\[
		\begin{aligned}
			\frac{(-1)^{\mathrm{q}(H_n)+\mathrm{q}(P_{\circ})}}{\mathrm{r}(H_n)\mathrm{r}(P_{\circ})}\cdot(\omega_{H^{\vee}_n}\wedge\omega_{H_n}\wedge\omega'_{P^{\vee}_{\circ}})=C_{n,n}\cdot\mathrm{c}_{H_n}^{H_n-1,P_{\circ}+1}\cdot\mathfrak{o}_{n,n}.
		\end{aligned}
		\]
		By \cite[Proposition 2.13(ii)]{IM} and \eqref{product}, we have
		\[
		\mathrm{c}^{H_n-1,P_{\circ}+1}_{H_n}=\mathrm{r}(P_{\circ}+1)^{-1}\cdot\mathrm{C}^{\circ}=\mathrm{r}(P_{\circ})^{-1}\cdot\mathrm{C}^{\circ},
		\]
		where
		\[
		\begin{aligned}
			\mathrm{C}^{\circ}=	&\mathrm{C}^{\circ}\left((n-1,n-3,\cdots,1-n);(n-2,n-4,\cdots,-n)\right)\\
			=&\prod_{1\leq i<j\leq n}\frac{3j-3i-1}{3j-3i+1}
		\end{aligned}
		\] 
		is computed by \cite[(2.37)]{IM}. Using the formulas \eqref{r} and \eqref{q}, a straightforward computation shows that
		\[
		\mathrm{r}(H_n)=1,\qquad\mathrm{q}(H_n)=\frac{n(n-1)(n+1)}{6},\qquad\mathrm{q}(P_{\circ})=0,
		\]
		which proves the lemma.
	\end{proof}
	
	It remains to calculate the wedge product $\omega_{H^{\vee}_n}\wedge\omega_{H_n}\wedge\omega'_{P^{\vee}_{\circ}}$. 
	
	\begin{lemp}\label{lemB}
		We have that
		\[
		\mathfrak{o}_{n,n}=\omega_{H^{\vee}_n}\wedge\omega_{H_n}\wedge\left(e_{1}^{\vee}\wedge e_{2}^{\vee}\cdots\wedge e_{n-1}^{\vee}\right).
		\]
	\end{lemp}
	
	\begin{proof}
		Recall that
		\[
		\begin{aligned}
			\omega_{H_n}=&\,\omega_n=\left(\bigwedge_{2\leq i\leq n}\bigwedge_{1\leq j\leq i-1}e_{i,j}^{\vee}\right)\\
			=&\,e_{2,1}^{\vee}\wedge e_{3,1}^{\vee}\wedge e_{3,2}^{\vee}\wedge\cdots\wedge e_{n,1}^{\vee}\wedge e_{n,2}^{\vee}\wedge\cdots\wedge e_{n,n-1}^{\vee}
		\end{aligned}
		\]
		is the highest weight vector of $\RI_n(\Xi_n)$ and we calculate the lowest weight vector
		\[
		\begin{aligned}
			\omega_{H_n^{\vee}}=&\,(-1)^{\mathrm{q}(H_n)}\\
			&\cdot e_{n-1,n}^{\vee}\wedge e_{n-2,n}^{\vee}\wedge e_{n-2,n-1}^{\vee}\wedge\cdots\wedge e_{1,n}^{\vee}\wedge e_{1,n-1}^{\vee}\wedge\cdots\wedge e_{1,2}^{\vee}\\
			=&\,(-1)^{\frac{n(n-1)(n+1)}{6}}\cdot(-1)^{\frac{(n-2)(n-1)n(n+3)}{12}}\\
			&\cdot e_{1,2}^{\vee}\wedge e_{1,3}^{\vee}\wedge e_{2,3}^{\vee}\wedge\cdots\wedge e_{1,n}^{\vee}\wedge e_{2,n}^{\vee}\wedge\cdots \wedge e_{n-1,n}^{\vee}\\
			=&(-1)^{\frac{n(n-1)(n+1)}{6}}\cdot(-1)^{\frac{(n-2)(n-1)n(n+3)}{12}}\cdot\left(\bigwedge_{2\leq j\leq n}\bigwedge_{1\leq i\leq j-1}e_{i,j}^{\vee}\right).
		\end{aligned}
		\]
		Hence,
		\[
		\begin{aligned}
			\omega_{H_n^{\vee}}\wedge\omega_{H_n}=&\,(-1)^{\frac{n(n-1)(n+1)}{6}}\cdot(-1)^{\frac{(n-2)(n-1)n(n+3)}{12}}\\
			&\cdot\left(\bigwedge_{2\leq j\leq n}\bigwedge_{1\leq i\leq j-1}e_{i,j}^{\vee}\right)\wedge\left(\bigwedge_{2\leq i\leq n}\bigwedge_{1\leq j\leq i-1}e_{i,j}^{\vee}\right)\\
			=&\,(-1)^{\frac{n(n-1)(n+1)}{6}}\cdot(-1)^{\frac{(n-2)(n-1)n(n+3)}{12}}\cdot(-1)^{\frac{n^2(n-1)^2}{4}}\\
			&\cdot \left(\bigwedge_{2\leq i\leq n}\bigwedge_{1\leq j\leq i-1}e_{i,j}^{\vee}\right)\wedge\left(\bigwedge_{2\leq j\leq n}\bigwedge_{1\leq i\leq j-1}e_{i,j}^{\vee}\right).
		\end{aligned}
		\]
		Comparing with the definition of $\mathfrak{o}_{n,n}$, the lemma follows from the fact
		\[
		(-1)^{\frac{n(n-1)(n+1)}{6}}\cdot(-1)^{\frac{(n-2)(n-1)n(n+3)}{12}}\cdot(-1)^{\frac{n^2(n-1)^2}{4}}=1.
		\]
	\end{proof}
	
	Recall that $\{\omega'_M\}_{M\in\check{\RG}'_n}$ is the image of the Gelfand-Tsetlin basis $\{\xi_M\}_{M\in\check{\RG}'_n}$ of $\Xi_n'^{\vee}$ under the inclusion $\RI_n':\Xi_n'^{\vee}\hookrightarrow\wedge^{c_n}\widetilde{\mathfrak{p}}_n^{\vee}$. Denote by $\{\varepsilon^{\iota}_{i,j}\}_{1\leq i,j\leq n}$ a basis of $\mathfrak{gl}_n(\C_{\iota})$. The explicit formulas for the action of $\varepsilon^{\iota}_{i,j}$ on the Gelfand-Tsetlin basis are given by \cite[(2,19), (2,20), (2,21)]{IM}. We record the following one:
	\[
	\varepsilon^{\iota}_{j,j+1}\xi_M=\sum_{\substack{1\leq i\leq j\\M+\Delta_{i,j}\in\check{\RG}_n'}}\mathrm{a}_{i,j}(M)\xi_{M+\Delta_{i,j}},\qquad M\in\check{\RG}_n',\, 1\leq j\leq n-1,
	\]
	where
	\[
	\Delta_{i,j}=(\delta_{i',j'})_{1\leq i'\leq j'\leq n},\qquad\delta_{i',j'}=\begin{cases}
		1, & i'=i,j'=j,\\
		0, &\text{otherwise},
	\end{cases}
	\]
	and
	\[
	\mathrm{a}_{i,j}(M)=\frac{\prod_{h=1}^i(m_{h,j+1}-m_{i,j}-h+i)}{\prod_{h=1}^{i-1}(m_{h,j}-m_{i,j}-h+i)}\left(\prod_{h=2}^i\frac{m_{h-1,j-1}-m_{i,j}-h+i}{m_{h-1,j}-m_{i,j}-h+i}\right).
	\]

	\begin{lemp}\label{lemC}
		We have that
		\[
		\xi_{P_{\circ}^{\vee}}=\frac{1}{n!(n-1)!\cdots 2!}\cdot(\varepsilon^{\iota}_{n-1,n})^1(\varepsilon^{\iota}_{n-2,n-1})^2\cdots(\varepsilon^{\iota}_{1,2})^{n-1}\xi_{H'^{\vee}_n}.
		\]
	\end{lemp}
	
	\begin{proof}
		Note that
		\[
		P_{\circ}^{\vee}=H_n'^{\vee}+\sum_{j=1}^{n-1}(n-j)\Delta_{j,j}.
		\]
		From the explicit formula above for the action of $\varepsilon^{\iota}_{j,j+1}$ ($1\leq j\leq n-1$), we calculate that
		\[
		\begin{aligned}
			&\,(\varepsilon^{\iota}_{n-1,n})^1(\varepsilon^{\iota}_{n-2,n-1})^2\cdots(\varepsilon^{\iota}_{1,2})^{n-1}\xi_{H'^{\vee}_n}\\
			=&\,\prod_{j=1}^{n-1}\prod_{k=0}^{n-1-j}\mathrm{a}_{j,j}\left(H'^{\vee}_n+\sum_{i=1}^{j-1}(n-i)\Delta_{i,i}+k\Delta_{j,j}\right)\xi_{P_{\circ}^{\vee}}.
		\end{aligned}
		\]
		Here the $M+\Delta_{i,j}$ ($i<j$) terms vanish due to the definition of the Gelfand-Tsetlin pattern. The lemma follows by
		\[
		\prod_{k=0}^{n-1-j}\mathrm{a}_{j,j}\left(H'^{\vee}_n+\sum_{i=1}^{j-1}(n-i)\Delta_{i,i}+k\Delta_{j,j}\right)=(n+1-j)!
		\]
		from the formula of $\mathrm{a}_{j,j}$ ($1\leq j\leq n-1$). 
	\end{proof}
	
	Recall the basis $\{e^{\vee}_{i,j}\}_{1\leq i,j\leq n}$ of $\mathfrak{p}^{\vee}_n$. Put
	\[
	e_{i,j}':=\varepsilon_{i,j}^{\iota}-\varepsilon_{j,i}^{\overline{\iota}},\qquad 1\leq i,j\leq n.
	\]
	Then $\{e_{i,j}'\}_{1\leq i,j\leq n}$ forms a basis of $\mathfrak{k}_n$. By \cite[Lemma 3.2(i)]{HMN}, the (adjoint) action of $\mathfrak{k}_n$ on $\mathfrak{p}_n^{\vee}$ is given by the formula
	\begin{equation}\label{action}
	e_{a,b}'.e_{i,j}^{\vee}=\delta_{j,b}e_{i,a}^{\vee}-\delta_{i,a}e_{b,j}^{\vee},\qquad 1\leq a,b,i,j\leq n,
	\end{equation}
	where
	\[
	\delta_{i,j}:=\begin{cases}
		1, & \text{if }i=j,\\
		0, & \text{otherwise},
	\end{cases}
	\]
	is the Kronecker symbol.

	\begin{lemp}\label{lemD}
		We have that
		\[
		\omega_{H_n^{\vee}}\wedge\omega_{H_n}\wedge\omega'_{P^{\vee}_{\circ}}=\frac{(-1)^{\frac{n(n-1)}{2}}}{(n-1)!}\cdot\mathfrak{o}_{n,n}.
		\]
	\end{lemp}
	
	\begin{proof}
We write
\[
\omega_{P_{\circ}^{\vee}}'=C_{1,2,\dots,n-1}^{1,2,\dots,n-1}e_1^{\vee}\wedge e_2^{\vee}\wedge\cdots\wedge e_{n-1}^{\vee}+\text{(other terms)},
\]
where `other terms' indicates a linear combination of wedge monomials in the basis
\[
\left\{e_i^{\vee}\right\}_{1\leq i\leq n-1}\sqcup\left\{e_{i,j}^{\vee}\right\}_{1\leq i\neq j\le n},
\]
each containing at least one off-diagonal factor. In view of Lemma \ref{lemB}, we have 
		\[
		\omega_{H^{\vee}_n}\wedge\omega_{H_n}\wedge\omega'_{P^{\vee}_{\circ}}=C_{1,2,\dots,n-1}^{1,2,\dots,n-1}\cdot\mathfrak{o}_{n,n}.
		\]
		
		By Lemma \ref{lemC}, we have that
		\[
		\omega'_{P_{\circ}^{\vee}}=\frac{1}{n!(n-1)!\cdots 2!}\cdot(e'_{n-1,n})^1(e'_{n-2,n-1})^2\cdots(e'_{1,2})^{n-1}\omega'_{H'^{\vee}_n}.
		\]
		Recall that $\RI_n'(\Xi'^{\vee}_n)$ has the highest weight vector
		\[
		\omega'_{\check{H}_n'}=\omega_n'=e_{n,1}^{\vee}\wedge e_{n,2}^{\vee}\wedge\cdots\wedge e_{n,n-1}^{\vee}
		\]
		and it is easy to obtain the lowest weight vector
		\[
		\omega'_{H'^{\vee}_n}=(-1)^{n-1}e_{1,2}^{\vee}\wedge e_{1,3}^{\vee}\wedge\cdots\wedge e_{1,n}^{\vee}.
		\]
		
		Using \eqref{action}, we compute that
		\[
		\begin{aligned}
			&\,(e'_{n-1,n})^1(e'_{n-2,n-1})^2\cdots(e'_{1,2})^{n-1}\omega'_{H'^{\vee}_n}\\
			=&\,(e'_{n-1,n})^1(e'_{n-2,n-1})^2\cdots(e'_{1,2})^{n-1}\left((-1)^{n-1}e_{1,2}^{\vee}\wedge e_{1,3}^{\vee}\wedge\cdots\wedge e_{1,n}^{\vee}\right)\\
			=&\,(-1)^{\frac{n(n-1)}{2}}\cdot n\cdot (n-1)!(n-2)!\cdots2!\cdot\left(e_{1}^{\vee}\wedge e_{2}^{\vee}\wedge\cdots\wedge e_{n-1}^{\vee}\right)\\
			&+\text{(other terms)},
		\end{aligned}
		\]
		where `other terms' has the meaning specified above. Hence
		\[
		C_{1,2,\dots,n-1}^{1,2,\dots,n-1}=\frac{(-1)^{\frac{n(n-1)}{2}}}{(n-1)!},
		\]
		which completes the proof of the lemma.
	\end{proof}
	
	Combining Lemmas \ref{lemA} and \ref{lemD}, we see that
	\[
	C_{n,n}=(-1)^{\frac{n(n-1)(n+1)}{6}}\cdot(\RC^{\circ})^{-1}\cdot\frac{(-1)^{\frac{n(n-1)}{2}}}{(n-1)!},
	\]
	which implies \eqref{CNN}.

	\subsection{Calculation of archimedean Rankin-Selberg integrals}
	\label{sec:4.4}
	
	From \cite[Corollary 2.16]{IM} and \cite[Lemma 2.14]{HMN}, it is routine to verify that
	\[
	\begin{aligned}
		&\sum_{M\in\RG_n}\sum_{M'\in\RG_n}\sum_{P\in\RG_n'}\frac{(-1)^{\mathrm{q}(M')}}{\mathrm{r}(M')}\mathrm{c}_{M'}^{M-1,P+1}\cdot\RZ^{\circ}_{0_n,0_n,\chi_0}\left(f_M,f_{M'^{\vee}},\varphi_P;\chi_0,\mathrm{d}^{\circ}g\right)\\
		=& \,(-1)^{\frac{n(n-1)(n+1)}{6}}\cdot\mathrm{i}^{b_n}\cdot n!\cdot\mathrm{C}^{\circ}.
	\end{aligned}
	\]
	We finally conclude from \eqref{reduction} and \eqref{CNN} that
	\[
	\begin{aligned}
		&\wp_{0_n,0_n,\chi_0}\left([\kappa_{0_n}]\otimes[\kappa_{0_n}]\otimes[\kappa_{\eta_0,\chi_0}]\otimes\mathbf{o}_{n,n}\right)\\
		=&\,\frac{4}{n}\cdot2^{-n(n+1)}\cdot\mathrm{i}^{-b_n}\cdot\frac{(-1)^{\frac{n(n-1)(n-2)}{6}}}{(n-1)!\cdot\mathrm{C}^{\circ}}\cdot(-1)^{\frac{n(n-1)(n+1)}{6}}\cdot\mathrm{i}^{b_n}\cdot n! \cdot \mathrm{C}^{\circ}\\
		=&\,2^{2-n(n+1)}\cdot(-1)^{\frac{n(n-1)}{2}},
	\end{aligned}
	\]
	which completes the proof of Theorem \ref{mainthm}(2).
	
	\subsection{Complex conjugation}\label{sec:complexconjugation}
	
	We assume that \eqref{CM'} holds and explain how our discussion applies to this case. The archimedean modular symbol $\wp_{\mu,\nu,\chi}$ is defined in the same way as \eqref{wpmunuchi}. Denote by $\eta_0$ the trivial character and $\chi_0:=\overline{\iota}|_{\K^{\times}}$ as before, so that $\overline{\chi}_0=\iota|_{\K^{\times}}$. For dominant weights $\mu=(\mu^{\iota};\mu^{\overline{\iota}})$ and $\nu=(\nu^{\iota};\nu^{\overline{\iota}})$, we write $\overline{\mu}=(\mu^{\overline{\iota}};\mu^{\iota})$ and $\overline{\nu}=(\nu^{\overline{\iota}};\nu^{\iota})$ respectively.

	By \cite[Lemma 4.3]{DX}, $\Xi_n'$ occurs with multiplicity one in $\wedge^{c_n}\widetilde{\mathfrak{p}}_n^{\vee}$. Moreover,
	\[
	\overline{\omega}_n':=\bigwedge_{1\leq j\leq n-1}e_{j,n}^{\vee}=e_{1,n}^{\vee}\wedge e_{2,n}^{\vee}\wedge\cdots\wedge e_{n-1,n}^{\vee}
	\]
	is a lowest weight vector of the $K_n$-type $\Xi_n'$ in $\wedge^{c_n}\widetilde{\mathfrak{p}}_n^{\vee}$. We fix an inclusion $\overline{\RI}'_n:\Xi'_n\hookrightarrow\wedge^{c_n}\widetilde{\mathfrak{p}}_n^{\vee}$ such that $\overline{\RI}'_n(\xi_{(\check{H}_n')^\vee})=\overline{\omega}_n'$ and denote $\overline{\omega}_M':=\overline{\RI}_n'(\xi_M)$ for each $M\in\RG_n'$.
	
	By \cite[Lemma 4.1]{DX}, the $K_n$-representation $\Xi_n'^{\vee}$ occurs with multiplicity one in $I_{\eta_0,\overline{\chi}_0}$. We fix an inclusion $\overline{\RJ}_n':\Xi'^{\vee}_n\hookrightarrow I_{\eta_0,\overline{\chi}_0}$ such that $\overline{\RJ}_n'(\xi_{\check{H}_n'})(1_n)=1$ and denote $\overline{\varphi}_M:=\overline{\RJ}_n'(\xi_M)$ for each $M\in\check{\RG}_n'$. As in Section \ref{induced}, we fix a generator
	\[
	[\kappa_{\eta_0,\overline{\chi}_0}]=\sum_{M\in\check{\RG}_n'}\frac{(-1)^{\mathrm{q}(M)}}{\mathrm{r}(M)}\overline{\omega}'_{M^{\vee}}\otimes\overline{\varphi}_M\in\RH^{c_n}\left(\mathfrak{g}_n,\widetilde{K}_n;I_{\eta_0,\overline{\chi}_0}\right),
	\]
	and
	\[
	[\kappa_{\eta,\chi}]:=\jmath_{\eta,\chi}([\kappa_{\eta_0,\overline{\chi}_0}])\in\RH^{c_n}\left(\mathfrak{g}_n,\widetilde{K}_n;F_{\eta,\chi}^{\vee}\otimes I_{\eta,\chi}\right),
	\]
	where $\jmath_{\eta,\chi}$ is normalized by the construction of \cite[Proposition 1.2]{JLSa} with $\iota$ and $\overline{\iota}$ interchanged.
    
	Recall the generators 
	\[
	[\kappa_{\mu}]\in\RH^{b_{n}}\left(\mathfrak{g}_{n},\widetilde{K}_{n};F_{\mu}^{\vee}\otimes\pi_{\mu}\right),\quad
	[\kappa_{\nu}]\in\RH^{b_{n}}\left(\mathfrak{g}_{n},\widetilde{K}_{n};F_{\nu}^{\vee}\otimes\pi_{\nu}\right)
	\]
	defined in Section \ref{generic}. We repeat the previous discussion to prove the following theorem.
	
	\begin{thmp}
		Retain the notations and assumptions as above. We have that
		\[
		\begin{aligned}
			&\wp_{\mu,\nu,\chi}\left([\kappa_{\mu}]\otimes[\kappa_{\nu}]\otimes[\kappa_{\eta,\chi}]\otimes\mathbf{o}_{n,n}\right)\\
			=&\,2^{2-n(n+1)}\cdot (-1)^{\frac{n(n-1)}{2}}\cdot  c'_{\overline{\mu},\overline{\nu},\overline{\chi}}\cdot\varepsilon'_{\overline{\mu},\overline{\nu},\overline{\chi}}.
		\end{aligned}
		\]
	\end{thmp}

	\begin{proof}
		Due to the archimedean period relation \cite[Theorem 1.6]{JLSa}, it suffices to compute
		\[
		\wp_{0_n,0_n,\overline{\chi}_0}\left([\kappa_{0_n}]\otimes[\kappa_{0_n}]\otimes[\kappa_{\eta_0,\overline{\chi}_0}]\otimes\mathbf{o}_{n,n}\right).
		\]
		
		As in Section \ref{sec:4.2}, there exists a constant $C_{n,n}'\in\C$ such that 
		\begin{equation}\label{cnn'}
			\begin{aligned}
				\wedge\left([\kappa_{0_n}]\otimes\right.&\left.[\kappa_{0_n}]\otimes[\kappa_{\eta_0,\overline{\chi}_0}]\right)=C'_{n,n}\cdot\mathfrak{o}_{n,n}\\
				\otimes&\left(\sum_{M\in\RG_n}\sum_{M'\in\RG_n}\sum_{P\in\check{\RG}_n'}\frac{(-1)^{\mathrm{q}(M')+\mathrm{q}(P)+b_n}}{\mathrm{r}(M)}\mathrm{c}_{M+1}^{M',P^{\vee}+1}\cdot f_M\otimes f_{M'^{\vee}}\otimes\overline{\varphi}_{P}\right)
			\end{aligned}
		\end{equation}
		by \cite[Lemma 4.4]{IM}, and we have that
		\[
		\begin{aligned}
			&\wp_{0_n,0_n,\overline{\chi}_0}\left([\kappa_{0_n}]\otimes[\kappa_{0_n}]\otimes[\kappa_{\eta_0,\overline{\chi}_0}]\otimes\mathbf{o}_{n,n}\right)\\
			=&\,\frac{4}{n}\cdot2^{-n(n+1)}\cdot\mathrm{i}^{-b_n}\cdot C'_{n,n}\cdot\\
			&\left(\sum_{M\in\RG_n}\sum_{M'\in\RG_n}\sum_{P\in\check{\RG}_n'}\frac{(-1)^{\mathrm{q}(M')+\mathrm{q}(P)+b_n}}{\mathrm{r}(M)}\mathrm{c}_{M+1}^{M',P^{\vee}+1}\cdot\RZ^{\circ}_{0_n,0_n,\overline{\chi}_0}\left(f_M,f_{M'^{\vee}},\overline{\varphi}_P;\overline{\chi}_0,\mathrm{d}^{\circ}g\right)\right).
		\end{aligned}
		\]
		
		The constant $C_{n,n}'$ can be calculated as in Section \ref{sec:4.3} by comparing the coefficients of the term
		\[
		f_{H_n}\otimes f_{H^{\vee}_n}\otimes\overline{\varphi}_{P^{\vee}_{\circ}}
		\]
		on both sides of \eqref{cnn'}. The same proof of Lemma \ref{lemA} shows that
		\[
		(-1)^{\frac{n(n-1)(n-2)}{6}}\cdot(\RC^{\circ})^{-1}\cdot	\omega_{H^{\vee}_n}\wedge\omega_{H_n}\wedge\overline{\omega}'_{P_{\circ}}=C'_{n,n}\cdot\mathfrak{o}_{n,n},
		\]
		and the same proof of Lemma \ref{lemD} shows that
		\[
		\omega_{H^{\vee}_n}\wedge\omega_{H_n}\wedge\overline{\omega}'_{P_{\circ}}=\frac{1}{(n-1)!}\cdot\mathfrak{o}_{n,n}.
		\]
		Therefore, we have that
		\[
		C'_{n,n}=\frac{(-1)^{\frac{n(n-1)(n-2)}{6}}}{(n-1)!\cdot\RC^{\circ}}.
		\]
		
		From \cite[Corollary 2.16]{IM} and \cite[Lemma 2.14]{HMN}, we have
		\[
		\begin{aligned}
			&\sum_{M\in\RG_n}\sum_{M'\in\RG_n}\sum_{P\in\check{\RG}_n'}\frac{(-1)^{\mathrm{q}(M')+\mathrm{q}(P)+b_n}}{\mathrm{r}(M)}\mathrm{c}_{M+1}^{M',P^{\vee}+1}\cdot\RZ^{\circ}_{0_n,0_n,\overline{\chi}_0}\left(f_M,f_{M'^{\vee}},\overline{\varphi}_P;\overline{\chi}_0,\mathrm{d}^{\circ}g\right)\\
			=& \,(-1)^{\frac{n(n-1)(n+1)}{6}}\cdot\mathrm{i}^{b_n}\cdot n!\cdot\mathrm{C}^{\circ}.
		\end{aligned}
		\]
		We conclude that
		\[
		\begin{aligned}
			\wp_{0_n,0_n,\overline{\chi}_0}\left([\kappa_{0_n}]\otimes[\kappa_{0_n}]\otimes[\kappa_{\eta_0,\overline{\chi}_0}]\otimes\mathbf{o}_{n,n}\right)=2^{2-n(n+1)}\cdot (-1)^{\frac{n(n-1)}{2}},
		\end{aligned}
		\]
		which completes the proof of the theorem.
	\end{proof}
	
	\section*{Acknowledgments}
    Y. Jin was supported by the National Natural Science Foundation of China No. 12601062. After the paper was written, ChatGPT-6 Astra was used to proofread the manuscript and verify computations, helping to identify and correct various errors in earlier versions.

\end{document}